\documentclass[letterpaper,11pt,reqno]{amsart}

\usepackage{mathtools}
\usepackage{amssymb}
\usepackage{mathrsfs}
\usepackage[dvipsnames]{xcolor}
\usepackage[
hypertexnames=false,
colorlinks=false,
linkcolor=NavyBlue,
filecolor=black,
urlcolor=blue,
citecolor=OrangeRed
]{hyperref}
\usepackage{aliascnt}
\usepackage[normalem]{ulem}

\numberwithin{equation}{section}

\newtheorem{theorem}{Theorem}[section]

\newaliascnt{proposition}{theorem}
\newtheorem{proposition}[proposition]{Proposition}
\aliascntresetthe{proposition}

\newaliascnt{lemma}{theorem}
\newtheorem{lemma}[lemma]{Lemma}
\aliascntresetthe{lemma}

\newaliascnt{corollary}{theorem}

\aliascntresetthe{corollary}

\theoremstyle{definition}
\newaliascnt{definition}{theorem}
\newtheorem{definition}[definition]{Definition}
\aliascntresetthe{definition}

\theoremstyle{remark}
\newaliascnt{remark}{theorem}
\newtheorem{remark}[remark]{Remark}
\aliascntresetthe{remark}

\usepackage[capitalize]{cleveref}
\usepackage{tikz}
\usetikzlibrary{arrows.meta, calc}
\crefname{theorem}{Theorem}{Theorems}
\Crefname{theorem}{Theorem}{Theorems}
\crefname{proposition}{Proposition}{Propositions}
\Crefname{proposition}{Proposition}{Propositions}
\crefname{lemma}{Lemma}{Lemmas}
\Crefname{lemma}{Lemma}{Lemmas}
\crefname{corollary}{Corollary}{Corollaries}
\Crefname{corollary}{Corollary}{Corollaries}
\crefname{definition}{Definition}{Definitions}
\Crefname{definition}{Definition}{Definitions}
\crefname{remark}{Remark}{Remarks}
\Crefname{remark}{Remark}{Remarks}

\newcommand{\R}{\mathbb R}
\newcommand{\Z}{\mathbb Z}

\newcommand{\op}{\operatorname}
\newcommand{\Ad}{\op{Ad}}
\newcommand{\diag}{\op{diag}}
\newcommand{\Fix}{\op{Fix}}
\newcommand{\vis}{\op{vis}}

\newcommand{\fa}{\mathfrak a}

\renewcommand{\L}{\mathcal L}
\newcommand{\La}{\Lambda}
\newcommand{\Ga}{\Gamma}

\newcommand{\LieA}{\mathfrak a}

\newcommand{\aL}{\mathfrak a}
\newcommand{\Fboundary}{\mathcal F}
\newcommand{\limitset}{\Lambda_\Gamma}
\newcommand{\limitcone}{\mathcal L_\Gamma}

\newcommand{\ret}{\mathsf K}
\newcommand{\gen}{\mathcal A}
\newcommand{\Sball}{\mathbf b}
\newcommand{\SBall}{\mathbf B}

\newcommand{\cA}{\mathcal A}
\newcommand{\g}{\mathfrak g}
\newcommand{\cnL}{\check{\mathfrak n}}
\newcommand{\mL}{\mathfrak m}
\newcommand{\nL}{\mathfrak n}
\newcommand{\F}{\mathcal F}

\renewcommand{\Ad}{\operatorname{Ad}}

\renewcommand{\vis}{\operatorname{vis}}

\title[Non-arithmeticity of Jordan spectra]{A dynamical proof of non-arithmeticity \\of Jordan spectra}

\author{Hee Oh}
\address{Department of Mathematics, Yale University, New Haven, CT}
\email{hee.oh@yale.edu}
\thanks{Oh is partially supported by the NSF grant No. DMS-2450703.}

\author{Pratyush Sarkar}
\address{Simons Laufer Mathematical Sciences Institute (SLMath), 17 Gauss Way, Berkeley, CA 94720}
\curraddr{Department of Mathematics, University of Utah, 155 South 1400 East, Salt Lake City, UT 84112}
\email{p.sarkar@utah.edu}

\begin{document}

	\begin{abstract}
		We give a dynamical proof of Benoist's non-arithmeticity theorem for
		Jordan spectra of Zariski dense subgroups of connected semisimple real
		algebraic groups. After passing to a Zariski dense Schottky subgroup, we use the coding of the limit set to realize Jordan
		projections as periods of a vector-valued Busemann return map for an expanding map on the Furstenberg boundary. The key step is to prove that a suitable two-branch
		asymptotic discrepancy is not locally constant on the limit set.  We also
		show that the same criterion applies beyond Lie groups; in particular, it
		yields a direct density result for multiplier spectra of hyperbolic rational maps
		whose Julia set is not contained in a circle.
	\end{abstract}

	\maketitle
	
	\section{Introduction}
	In rank one, the length spectrum of a locally symmetric space consists of
	the lengths of its closed geodesics, or equivalently, the translation lengths
	of hyperbolic elements of the corresponding discrete group of isometries. The
	non-arithmeticity of the length spectrum asserts that the subgroup of
	\(\mathbb R\) generated by these lengths is dense; this was proved by Dal'bo \cite{Dalbo} for surfaces and by I. Kim \cite{Kim_inkang} in general.  This property is a
	fundamental input in many problems in dynamics and rigidity.

	The Jordan spectrum is the higher-rank analogue of the length spectrum. The corresponding higher-rank non-arithmeticity theorem was proved by Benoist (\cite{Benoist97}, \cite{Benoist_00}). It has since become an important structural result in higher-rank dynamics, especially in questions involving counting, transfer operators, horospherical invariant measure classification, rigidity and local mixing. The goal of this paper is to give an alternative proof of Benoist's theorem. Our proof is more dynamical in nature than Benoist's original proof and also differs from Quint's later proof \cite{Quint_Schottky}. 
	
	Let $G$ be a connected semisimple real algebraic group. Let $\LieA$ be the Lie algebra of a real split torus $A < G$ and fix a positive Weyl chamber $\LieA^+$.
	We denote by
	\[
	\lambda:G\to\LieA^+
	\]
	the Jordan projection. Thus, if $g=g_e g_h g_u$
	is the Jordan decomposition of $g$ as a commuting product of an elliptic element $g_e$, a hyperbolic element $g_h$, and a unipotent element $g_u$, then $\lambda(g)$ is the unique element of $\LieA^+$ such that $g_h$ is conjugate to $\exp(\lambda(g))$.
	For example, when \(G=\mathrm{SL}_n(\mathbb R)\), if $|\alpha_1|\geq \cdots \geq |\alpha_n|$
	are the moduli of the eigenvalues of \(g\), counted with multiplicity, then
	\[
	\lambda(g)=\diag(\log|\alpha_1|,\ldots,\log|\alpha_n|).
	\]

	Let $\Gamma<G$ be a Zariski dense  subgroup. The set
	\[
	\lambda(\Gamma)\subset\LieA^+
	\]
	is called the Jordan spectrum of $\Gamma$. It records the spectral data of elements of $\Gamma$.
	The smallest closed cone in $\LieA^+$ containing $\lambda(\Gamma)$ is called the limit cone of $\Ga$ and is denoted by $\limitcone$.
	
	The main contribution of this paper is a new proof of the non-arithmeticity
	of the Jordan spectrum. Together with Benoist's convexity theorem for the
	limit cone, this also recovers the nonempty interior of the limit cone.
	\begin{theorem}[Benoist, \cite{Benoist97}, \cite{Benoist_00}]
		\label{thm:Benoist-main}\label{m1}
		Let $G$ be a connected semisimple real algebraic group and let $\Gamma<G$ be Zariski dense. Then:
		\begin{enumerate}
			\item the Jordan spectrum $\lambda(\Gamma)\subset\LieA$ is non-arithmetic, i.e., it generates a dense subgroup of $\LieA$;
			\item the limit cone $\limitcone$ has nonempty interior.
		\end{enumerate}
	\end{theorem}

	Theorem \ref{m1} is a consequence of the following one-dimensional scalar strengthening of the non-arithmeticity statement.
	
	\begin{theorem}\label{m2}
		\label{thm:projected-Jordan-spectrum-non-arithmetic}
		Let $G$ be a connected semisimple real algebraic group and let $\Gamma<G$ be Zariski dense. Then, for every $\psi\in\LieA^*-\{0\}$, the subset
		\[
		\psi(\lambda(\Gamma))\subset\R
		\]
		generates a dense subgroup of $\R$.
	\end{theorem}

	Kim's proof \cite{Kim_inkang} in rank one
	used cross ratios on the geometric boundary to rule out
	the possibility that all translation lengths lie in a discrete subgroup of
	\(\mathbb R\).
	Benoist's proof is algebraic and proximal in nature. It uses Schottky semigroups, precise estimates for the Jordan projections of matrix products, and an interpolation argument based on Hardy fields \cite{Benoist97}. Quint \cite{Quint_Schottky} later gave an alternative approach based on asymptotic estimates for the structural defects $\lambda(g_nh_n)-\lambda(g_n)-\lambda(h_n)$, where $g_n$ and $h_n$ are certain sequences of words in two elements $g, h \in \Gamma$ of length $\ll n$. A related argument is given in the book of Benoist--Quint \cite{BenoistQuint16} based on the related notion of multicross-ratios $\lim_{n \to \infty} \bigl(\lambda(g^nh^n)-\lambda(g^n)-\lambda(h^n)\bigr)$ of so-called transversally loxodromic elements $g, h \in \Gamma$.

	In this paper, the problem is translated into smooth boundary dynamics for
	an expanding Schottky map.  After passing to a suitable Schottky subgroup, we code the limit set $\Lambda_\Gamma$ by infinite reduced words and realize the Jordan projections of $\Gamma$ as the periods of an $\aL$-valued Busemann return map
	\[
	\ret:\Lambda_\Gamma\to \aL
	\]
	for a uniformly expanding shift map. For a nonzero linear form $\psi\in\aL^*$, the problem is then reduced to ruling out a lattice constraint $c\Z$ on the periods of the scalar return function $\psi\circ \ret$. 
	
	The possible lattice obstruction is detected through a two-branch asymptotic
	discrepancy. Using Liv\v{s}ic theory, we show that, if the scalar return function
	\(\psi\circ K\) had periods contained in a discrete subgroup of \(\mathbb R\),
	then every such scalar discrepancy would be locally constant on the limit set.
	We therefore introduce a nonconstant discrepancy condition, requiring that for
	some pair of backward Schottky branches, the scalar discrepancy have a nonzero
	derivative at a limit point. The main geometric step is to prove this condition
	for every nonzero \(\psi\in\mathfrak a^*\).
	Geometrically, the discrepancy is the \(A\)-component of a rectangular holonomy
	obtained by alternating stable and unstable horospherical moves. Differentiating
	this holonomy reduces the required nonvanishing to a Lie bracket computation.
	The final contradiction uses the fact that the limit set is not locally contained
	in a lower-dimensional smooth submanifold.
	
	This reformulation places the qualitative non-arithmeticity problem in the same boundary-dynamical framework that underlies thermodynamic approaches to counting, renewal theorems, and local mixing for Anosov subgroups.

	Finally, we point out that the nonconstant discrepancy criterion is not
	specific to Lie groups. It applies more generally to expanding
	systems in which the relevant periodic data arise as periods of a return
	function and whose underlying invariant set satisfies a suitable nondegeneracy property. As an
	illustration, in \cref{subsec:rational-map-analogue},
	we give a direct proof of a qualitative density statement for multiplier spectra of hyperbolic rational maps. This statement is much weaker than the prime orbit theorem of Oh--Winter
	\cite{OW17}, and is also subsumed by the recent more general theorem  of
	Ji--Xie--Zhang \cite{JXZ26}, whose proof is arithmetic and algebraic in nature. The point of including it here is different: in the hyperbolic case, the conclusion follows from the same dynamical criterion as above, using only the nonconstant discrepancy condition and the local nondegeneracy of the Julia set. 
	
	\subsection*{Acknowledgements}
	Part of this research was performed while the authors were visiting the Mathematical Sciences Research Institute (MSRI), now becoming the Simons Laufer Mathematical Sciences Institute (SLMath), which is supported by the National Science Foundation (Grant No. DMS-2424139).
	
	\section{Preliminaries}
	
	Let $G$ be a connected semisimple real algebraic group and $\g = \operatorname{Lie}(G)$ be its Lie algebra.
	Fix an Iwasawa decomposition
	\[
	G=KAN,
	\]
	and fix a minimal parabolic subgroup with Langlands decomposition
	\[
	P=MAN.
	\]
	Here $K$ is a maximal compact subgroup,
	$A$ is (the identity component of) a maximal real split torus, $M$ is a compact group commuting with $A$, and $N$ is a maximal horospherical subgroup.
	Let $\check P=MA\check N$ be the opposite minimal parabolic, where ${\check N}$ is a maximal horospherical subgroup opposite to $N$.
	Then $P\cap {\check P}=MA$. If $w_0$ denotes a representative of the longest Weyl element, then
	\[
	\check P=w_0Pw_0^{-1}.
	\]
	
	We write $\fa=\operatorname{Lie}(A)$, $\mathfrak n=\operatorname{Lie}(N)$, $\check{\mathfrak n}=\operatorname{Lie}(\check N)$, and $\mathfrak m=\operatorname{Lie}(M)$ for the corresponding Lie algebras. We use the vector space decomposition
	\[
	\g=\cnL\oplus \aL\oplus \mL\oplus \nL
	\]
	and denote by
	\[
	\pi_{\aL}:\g\to\aL
	\]
	the corresponding projection.

	Fix the unique positive Weyl chamber $\aL^+$ such that $N$ is contracting, i.e., as $t\to +\infty$,
	\[
	\exp(-tv)n\exp(tv) \to  e,
	\qquad
	n \in N,
	\quad
	v \in \operatorname{int}(\LieA^+).
	\]
	Let $\Pi \subset \aL^*$ be the associated set of simple roots. We write $\g_\alpha$ for the root space of any root $\alpha$. For $\alpha\in\Pi$, let $H_\alpha\in\aL$ be the vector dual to $\alpha$ under the Killing form $B$ restricted to $\aL$, i.e.,
	\[
	B(H,H_\alpha)=\alpha(H),\qquad H\in\aL;
	\]
	or equivalently,
	$\R H_\alpha$ is the $\aL$-component of $[\mathfrak g_\alpha,\mathfrak g_{-\alpha}]$. Note that $\{H_\alpha\}_{\alpha \in \Pi}$ forms a basis of $\LieA$.

	The Furstenberg boundary is given by
	\[
	\Fboundary:=G/P \cong K/M.
	\]
	For $g\in G$, set its forward and backward endpoints
	\[
	g^+:=gP\in\Fboundary,
	\qquad
	g^-:=gw_0 P\in \Fboundary.
	\]
	The open $G$-orbit in $\Fboundary\times\Fboundary$ is denoted by
	\[
	\Fboundary^{(2)}=\{(g P, gw_0P)\in \Fboundary\times \Fboundary: g\in G\}.
	\] We say that $\zeta$ and $\eta$ are in general position if $(\zeta, \eta)\in \Fboundary^{(2)}$.
	
	For a Zariski dense subgroup $\Gamma<G$, there exists a unique $\Ga$-minimal subset
	in $\Fboundary$ called the limit set of $\Gamma$ (cf. \cite{Benoist97}), which we denote by
	\[
	\limitset\subset\Fboundary.
	\]

	\section{Schottky coding of the limit set}
	Our use of Schottky coding follows a line of ideas going back, in rank one,
	to Lalley's symbolic and renewal-theoretic approach to Schottky groups
	\cite{Lal89}, and in higher rank, to Quint's work on Schottky subgroups of
	semisimple groups \cite{Quint_Schottky}. In this framework, finite reduced words
	encode group elements, infinite reduced words encode limit points, and
	deleting the first letter gives an expanding map on the limit set. 
	The Schottky contraction property gives uniformly contracting inverse
	branches. 
	
	The following definition is a slight modification of the definition given in \cite{Quint_Schottky}. 
	Fix a Riemannian metric $d_\F$ on $\F$, say the one induced by a bi-invariant Riemannian metric on $K$.
	\begin{definition}[Schottky generating set]
		\label{def:Schottky}
		A nonempty finite symmetric subset \(\mathcal A\subset G\) is called a
		\emph{Schottky generating set} if there exist \(\tau\in(0,1)\) and open subsets
		\[
		\{\Sball_\gamma\subset\Fboundary\}_{\gamma\in\mathcal A},
		\qquad
		\{\SBall_\gamma\subset\Fboundary\}_{\gamma\in\mathcal A},
		\]
		such that:
		\begin{enumerate}
			\item \(\overline{\Sball_\gamma}\subset\SBall_\omega\) for all
			\(\gamma,\omega\in\mathcal A\) with \(\gamma\neq\omega^{-1}\);
			
			\item \(\Sball_\gamma\times\Sball_\omega\subset\Fboundary^{(2)}\) for all
			\(\gamma,\omega\in\mathcal A\) with \(\gamma\neq\omega\);
			
			\item the restriction \(\gamma|_{\SBall_\gamma}\) is \(\tau\)-Lipschitz and
			\[
			\gamma\SBall_\gamma\subset\Sball_\gamma
			\]
			for all \(\gamma\in\mathcal A\);
			
			\item
			\[
			\bigcap_{\gamma\in\mathcal A}\SBall_\gamma\neq\varnothing .
			\]
		\end{enumerate}
		The subgroup \(\langle\mathcal A\rangle\) is called a \emph{Schottky subgroup}.
	\end{definition}
	
	We have the following lemma (cf. \cite[Proposition 4.3]{Benoist97}, \cite{ELO_anosov}).
	
	\begin{lemma}
		\label{elo2} Every Zariski dense subgroup of \(G\) contains a Zariski dense Schottky
		subgroup.
	\end{lemma}
	
	Until the final proof of \cref{m2}, we work with a fixed Zariski dense
	Schottky subgroup
	\[
	\Gamma=\langle\mathcal A\rangle<G
	\]
	generated by a finite symmetric Schottky generating set \(\mathcal A \subset G\).
	
	A finite word $\gamma=\gamma_0\gamma_1\cdots\gamma_{n-1} \in \Gamma$, where $\gamma_i\in\gen$,
	is called \emph{reduced} if $\gamma_{i+1}\neq\gamma_i^{-1}$ for all $i$. We use the convention that the empty word gives the identity element $e \in \Gamma$.
	We denote by $\gen^{(n)} \subset \Gamma$ the set of reduced words of length $n \geq 0$ in the alphabet $\gen$; thus $\gen^{(0)}=\{e\}$ and $\gen^{(1)}=\gen$.
	For a nontrivial reduced word $\gamma=\gamma_0\gamma_1\cdots\gamma_{n-1}$, write
	\[
	\sigma(\gamma):=\gamma_1\gamma_2\cdots\gamma_{n-1}
	\]
	for the word obtained by deleting the first letter.
	For each reduced word $\gamma=\gamma_0\cdots\gamma_{n-1}$ of length $n\geq1$, define the associated $n$th-level Schottky domain by
	\[
	\Sball_\gamma:=\gamma_0\gamma_1\cdots\gamma_{n-2}\Sball_{\gamma_{n-1}};
	\]
	or equivalently,
	\[
	\Sball_\gamma=\gamma_0\Sball_{\sigma(\gamma)}.
	\]
	If the concatenation $\gamma\omega$ of two reduced words $\gamma$ and $\omega$ is itself reduced, then
	\[
	\gamma\Sball_\omega=\Sball_{\gamma\omega}.
	\]
	For $n\geq1$, set
	\[
	\Sball^{(n)}:=\bigsqcup_{\gamma\in\gen^{(n)}}\Sball_\gamma.
	\]
	For $\gamma \in \gen$, it will be convenient to denote
	\[
	{\mathbf b}_\gamma^-:={\mathbf b}^{(1)}- {\mathbf b}_\gamma.
	\]
	
	For an infinite reduced word $\xi=\xi_0\xi_1\cdots$ and $m \geq 0$, write
	\[
	\xi_{[m]}:=\xi_0\xi_1\cdots\xi_{m-1}\in\gen^{(m)}.
	\]
	Fix a point
	\[
	o\in\bigcap_{\gamma\in\gen}\SBall_\gamma.
	\]
	Every infinite reduced word $\xi=\xi_0\xi_1\cdots$
	determines a point of the limit set by
	\[
	\xi^+:=\lim_{k\to\infty}\xi_{[k + 1]} o=\lim_{k\to\infty}\xi_0\xi_1\cdots\xi_k o.
	\]
	This limit is independent of the choice of $o$ and is equivalently given by
	\[
	\{\xi^+\}=\bigcap_{k\geq0}\xi_{[k]}\Sball_{\xi_k}=\bigcap_{k\geq0}\xi_0\xi_1\cdots\xi_{k-1}\Sball_{\xi_k}.
	\]
	This gives a  homeomorphism from the space $\Sigma^+$ of infinite reduced words, with its usual topology, onto $\La_\Ga$ \cite[Proposition 3.3]{Quint_indicator}.

	\begin{remark}
		It follows from definitions that every nontrivial element of $\Gamma$ is loxodromic and any two distinct points of $\La_\Gamma$ are in general position.
	\end{remark}

	We use the same notation $d_{\mathcal F}$ for its restriction to $\Lambda_\Gamma$. The Schottky contraction
	property implies that if two infinite reduced words have a common prefix of
	length \(N\), then their corresponding points in \(\limitset\) are
	\(O(\tau^N)\)-close with respect to \(d_{\mathcal F}\). 
	
	For some \(\vartheta \in (0, 1)\), we also use the symbolic metric \(d_\vartheta\) on $\Sigma^+$, defined by
	\[
	d_\vartheta(\xi,\zeta)=\vartheta^{N(\xi,\zeta)},
	\]
	where \(N(\xi,\zeta)\) is the first index at which
	\(\xi\) and \(\zeta\) differ. Hence the coding map from $(\Sigma^+, d_\vartheta)$ to
	\((\limitset,d_{\mathcal F})\) is bi-H\"older. Note that it can be made Lipschitz in the forward direction by choosing $\vartheta \geq \tau$.

	We shall use the same symbol $\xi$ for the infinite reduced word and the corresponding point $\xi^+\in\limitset$ when no confusion can arise.
	If $\gamma=\gamma_0\gamma_1\cdots\gamma_{n-1}\in\gen^{(n)}$
	is cyclically reduced, then the infinite word
	\[
	\xi:=\gamma\gamma\cdots\in\limitset
	\]
	is reduced and defines a point in $\La_\Ga$. This point is the
	attracting fixed point of \(\gamma\). Indeed, the Schottky contraction
	property implies that $\gamma^k z\xrightarrow{k \to \infty} \xi$
	for every \(z\in \SBall_{\gamma_{n-1}}\).

	\begin{definition}[Shift map]
		Deleting the first Schottky letter of infinite reduced words defines  the \emph{shift map}
		\[
		\sigma:\limitset\to\limitset,
		\qquad
		\sigma(\xi_0\xi_1\cdots)=\xi_1\xi_2\cdots.
		\] 
		We use the same notation for the extension
		\[
		\sigma:\Sball^{(n+1)}\to\Sball^{(n)},
		\]
		defined by
		\[
		\sigma(x)=\gamma^{-1}x,
		\qquad
		\text{for }x\in\gamma\bigl(\Sball^{(n)}-\Sball_{\gamma^{-1}}\bigr),
		\quad \gamma\in\gen.
		\] \end{definition}
	Thus the inverse branches are
	the maps \(x\mapsto \gamma x\), defined on
	\(\Sball^{(n)}-\Sball_{\gamma^{-1}}\),  and these are uniformly
	contracting by the Schottky property. Hence \(\sigma\) is expanding.

	\section{The Busemann return map and the nonconstant discrepancy condition}
	
	In this section, we first introduce the Busemann return map $\ret$ used throughout the proof. It is
	the contribution of the \(\LieA\)-valued Busemann function corresponding to the first Schottky
	letter removed by the shift map. Its periodic sums recover the Jordan
	projections of elements of \(\Gamma\). We then formulate a non-lattice
	criterion for scalar projections of $\ret$ as a certain geometric nondegeneracy statement for a two-branch asymptotic discrepancy. We finally show that it indeed implies the non-arithmeticity of the scalar projections of the Jordan spectrum $\psi(\lambda(\Gamma)) \subset \R$.
	
	\subsection{The Busemann return map}
	The Iwasawa cocycle
	\[
	\mathsf c:G\times\Fboundary\to\LieA
	\]
	is a smooth map defined as follows: for $g \in G$ and
	\(\xi=kP\in\Fboundary\) with \(k\in K\), there exists a unique element $c(g,\xi)$ satisfying the condition
	\[
	gk\in K\exp(\mathsf c(g,\xi))N .
	\]
	The $\LieA$-valued Busemann function $\beta: \Fboundary \times G \times G \to \LieA$ is defined as follows: for \(\xi\in\Fboundary\) and \(g,h\in G\), set
	\[
	\beta_\xi(g,h)
	:=
	\mathsf c (g^{-1},\xi)-\mathsf c(h^{-1},\xi);
	\]
	in particular,
	\[
	\beta_\xi(e,g)=-\mathsf c(g^{-1},\xi).
	\]
	It enjoys the following properties for all $\xi \in \Fboundary$ and $g, h, q \in G$:
	\begin{itemize}
		\item (additivity) $\beta_\xi(g, q) = \beta_\xi(g, h) + \beta_\xi(h, q)$;
		\item ($G$-equivariance) $\beta_{g\xi}(gh, gq) = \beta_\xi(h, q)$.
	\end{itemize}
	Note that additivity of the Busemann function follows by combining the cocycle identity inherited from $\mathsf c$ and $G$-equivariance.
	Moreover, if \(g \in G\) is loxodromic with attracting fixed point \(g^+\), then
	\[
	\beta_{g^+}(e,g)=\lambda(g)
	\]
	(cf. \cite[Lemma 3.5]{LeeOh23}).
	
	\begin{definition}[Busemann return map]
		\label{def:first-return-vector-map}
		The $\fa$-valued \emph{Busemann return map}
		\[
		\ret:\Sball^{(2)}\to\LieA
		\]
		is defined by
		\[
		\ret(x):=\beta_x(e,\gamma),
		\qquad
		\text{for }x\in\gamma\bigl(\Sball^{(1)}-\Sball_{\gamma^{-1}}\bigr),
		\quad \gamma\in\gen.
		\]
	\end{definition}
	
	Observe that for $\omega \in \gen - \{\gamma^{-1}\}$, the shift map \(\sigma\) sends $x\in\Sball_{\gamma\omega}$
	to  $\sigma x=\gamma^{-1}x\in\Sball_\omega$ and \(\ret(x)\) records the Busemann displacement associated to this one
	Schottky step.
	
	Note that \(\ret\) is smooth with finite $C^\ell$ norms for $\ell \geq 0$ since it is defined in terms of the Busemann function on each second-level Schottky domain, and therefore Lipschitz with respect to
	\(d_{\mathcal F}\). It follows that \(\ret|_{\Lambda_\Gamma}\)
	is H\"older continuous with respect to the symbolic metric \(d_\vartheta\).
	
	For $\psi \in \LieA^* - \{0\}$, we often consider the function ${{\mathsf K}^\psi}:=\psi \circ \ret$. We sometimes refer to it as a \emph{scalar return function}.
	
	For $n\geq0$, define
	$\ret_n:\Sball^{(n+1)}\to\LieA$
	by
	\[
	\ret_n(x):=\sum_{j=0}^{n-1}\ret(\sigma^j x),
	\]
	with the convention that $\ret_0=0$. We use similar notations as above for other functions on $\Sball^{(2)}$ or $\limitset$, such as for the scalar return function ${{\mathsf K}^\psi}$, or otherwise.
	
	\begin{lemma}
		\label{lem:Jordan-proj-in-terms-of-return-vector}
		Let $\gamma=\gamma_0\gamma_1\cdots\gamma_{n-1}\in\gen^{(n)}$ for $n \geq 1$
		be cyclically reduced, and let $ \xi:=\gamma\gamma\cdots\in\limitset$
		be the attracting fixed point of $\gamma$. Then
		\[
		\lambda(\gamma)=\ret_n(\xi).
		\]
	\end{lemma}
	
	\begin{proof}
		Since $\gamma$ is cyclically reduced, the infinite word $\xi$
		is reduced. Hence, for $0\leq j\leq n-1$, the point $\sigma^j\xi$ lies in $\Sball_{\gamma_j\gamma_{j+1}}$, where the indices are taken modulo $n$. Therefore
		\[
		\ret_n(\xi)=\sum_{j=0}^{n-1}\ret(\sigma^j\xi)=\sum_{j=0}^{n-1}\beta_{\sigma^j\xi}(e,\gamma_j).
		\]
		For $0\leq j\leq n-1$, we further have
		\[
		\sigma^j\xi=\xi_{[j]}^{-1}\xi=(\gamma_0\gamma_1\cdots\gamma_{j-1})^{-1}\xi,
		\]
		and so by the cocycle identity,
		\[
		\begin{aligned}
			\beta_\xi(e,\gamma)
			=\sum_{j=0}^{n-1}\beta_{(\gamma_0\cdots\gamma_{j-1})^{-1}\xi}(e,\gamma_j) =\sum_{j=0}^{n-1}\beta_{\sigma^j\xi}(e,\gamma_j).
		\end{aligned}
		\]
		Thus $\ret_n(\xi)=\beta_\xi(e,\gamma)$. Since $\xi$ is the attracting fixed point of $\gamma$, we have $\beta_\xi(e,\gamma)=\lambda(\gamma)$.
	\end{proof}
	
	\begin{remark}
		For an arbitrary nontrivial element of $\Gamma$, choose the cyclically reduced representative of its conjugacy class. Since the Jordan projection is conjugacy invariant, \cref{lem:Jordan-proj-in-terms-of-return-vector} accounts for the Jordan projections of all nontrivial elements of $\Gamma$.
	\end{remark}
	
	\subsection{The nonconstant discrepancy condition}
	\begin{definition}[Discrepancy]
		For an infinite reduced word $\xi=\xi_0\xi_1\cdots$, define the
		\emph{(asymptotic) discrepancy},
		\[
		\Delta_\xi: \Sball_{\xi_0}^- \times \Sball_{\xi_0}^-\to \LieA,
		\]
		as follows: for \(x,x'\in\Sball_{\xi_0}^-\), set
		\[
		\begin{aligned}
			\Delta_\xi(x,x')
			&:=
			\lim_{m\to\infty}
			\left(
			\ret_m(\xi_{[m]}^{-1}x)
			-
			\ret_m(\xi_{[m]}^{-1}x')
			\right) \\
			&=
			\sum_{k=1}^{\infty}
			\left(
			\ret(\xi_{[k]}^{-1}x)
			-
			\ret(\xi_{[k]}^{-1}x')
			\right).
		\end{aligned}
		\]
		For every $\psi \in \LieA^* - \{0\}$, we write $\Delta^\psi_\xi := \psi \circ \Delta_\xi$ and call it a \emph{scalar (asymptotic) discrepancy}.
	\end{definition}
	
	The scalar asymptotic discrepancy is often called the \emph{temporal distance function} in the literature. The above expression is well-defined and produces a $C^\infty$ function for the following reason.
	The Schottky contraction property implies that the inverse branches
	\[
	x\mapsto \xi_{[k]}^{-1}x
	\]
	are uniformly contracting on  
	\(\Sball_{\xi_0}^-\). Since for every $\ell \geq 0$, the Busemann return map \(\ret\) is smooth with a finite $C^\ell$ norm, the summands
	\[
	(x,x')\mapsto \ret(\xi_{[k]}^{-1}x)-\ret(\xi_{[k]}^{-1}x')
	\]
	decay exponentially in \(C^\ell\) as $k \to \infty$. More explicitly, for every $\ell \geq 0$, on $\Sball_{\xi_0}^-\times\Sball_{\xi_0}^-$, we have
	\[
	\left\| \ret(\xi_{[k]}^{-1}\bullet_1)-\ret(\xi_{[k]}^{-1}\bullet_2) \right\|_{C^\ell(\Sball_{\xi_0}^-\times\Sball_{\xi_0}^-)} \ll \tau^k \qquad\text{for all } k\geq 1.
	\]
	Consequently, the series defining \(\Delta_\xi\) converges in \(C^\infty\).

	\begin{definition}[Two-branch discrepancy]
		For two infinite reduced words $\zeta$ and $\eta$, define the
		\emph{two-branch (asymptotic) discrepancy},
		\[
		\Phi_{\zeta,\eta}:  (\Sball_{\zeta_0}^- \cap \Sball_{\eta_0}^-)\times  (\Sball_{\zeta_0}^- \cap \Sball_{\eta_0}^-) \to \LieA
		\]
		as follows: for all $x, x' \in \Sball_{\zeta_0}^- \cap \Sball_{\eta_0}^-$, set
		\[
		\Phi_{\zeta,\eta}(x,x'):=\Delta_\zeta(x,x')-\Delta_\eta(x,x').
		\]
		For every $\psi \in \LieA^* - \{0\}$, we write $\Phi^\psi_{\zeta,\eta} := \psi \circ \Phi_{\zeta,\eta}$ and call it a \emph{scalar two-branch (asymptotic) discrepancy}.
	\end{definition}
	
	Since $\Delta_\zeta$ and $\Delta_\eta$ are $C^\infty$, so is \(\Phi_{\zeta,\eta}\) on its domain. Here and in the rest of the paper,
	$D_x$ denotes the differential with respect to the first variable. More explicitly for $\Phi_{\zeta,\eta}$, for instance, we have the following: for $v\in T_{x_0}\Fboundary$,
	\[
	D_x\Phi_{\zeta,\eta}(x_0,x_0')(v)
	:=
	\left.\frac{d}{dt}\right|_{t=0}
	\Phi_{\zeta,\eta}(c(t),x_0'),
	\]
	where $c$ is any smooth curve in $\Fboundary$ satisfying $c(0)=x_0$ and $c'(0)=v$.
	
	\begin{definition}[Nonconstant discrepancy condition]
		\label{def:non-lattice-condition}
		Let $\psi\in\LieA^*-\{0\}$ and set
		\[
		{{\mathsf K}^\psi}:=\psi\circ\ret:\Sball^{(2)}\to\R.
		\]
		We say that the scalar return function ${{\mathsf K}^\psi}$ satisfies the \emph{nonconstant discrepancy condition} if there exist two infinite reduced words $\zeta$ and $\eta$, points
		$x_0,x_0'
		\in
		\limitset \cap \bigl(\Sball_{\zeta_0}^- \cap \Sball_{\eta_0}^-\bigr)$,
		and a tangent vector $v\in T_{x_0}\Fboundary$ such that
		\[
		D_x\Phi^\psi_{\zeta,\eta}(x_0,x_0')(v) = \psi\bigl(D_x\Phi_{\zeta,\eta}(x_0,x_0')(v)\bigr) \neq0.
		\]
	\end{definition}
	
	Geometrically, this condition says that, after applying \(\psi\), the
	two-branch asymptotic discrepancy is not locally constant in the first variable.

	\begin{remark}\label{nli2}
		The condition above is a qualitative version of the non-local integrability
		condition familiar from Dolgopyat-type arguments \cite{Dol98}. In the present paper we use
		the term ``nonconstant discrepancy condition'' to emphasize that we only need
		a single nonzero derivative of an asymptotic two-branch discrepancy. This is
		weaker than the quantitative finite-branch formulations of non-integrability
		used in transfer-operator estimates.
	\end{remark}

	For the proof of \cref{prop:non-lattice-criterion}, we need two lemmas. The first one is the Liv\v{s}ic theorem for a one-sided subshift of finite type reformulated in our Schottky setting. Recall that via the Schottky coding, \((\limitset,\sigma)\) is conjugate to a one-sided subshift of finite type.

	\begin{lemma}[Liv\v{s}ic theorem]
		\label{lem:Livsic-Schottky-shift}
		Let \(f:\limitset\to\mathbb R\) be H\"older continuous and let $ L=c\mathbb Z$
		for some \(c\geq 0\).
		Suppose that
		$$\{ f_n(\xi): \xi\in\Fix(\sigma^n), n\ge 1 \}\subset L $$
		where $f_n(\xi)=\sum_{j=0}^{n-1} f(\sigma^j(\xi)).$
		Then \(f\) is cohomologous to an $L$-valued function. More precisely, there exists
		a continuous function
		$\rho:\limitset\to\mathbb R/L$
		such that
		\[
		f=\rho-\rho\circ\sigma
		\qquad \pmod{L}.
		\]
	\end{lemma}

	For a proof of the above lemma, we refer the reader to \cite[Proposition 3.7]{ParryPollicott1990} for \(L=\{0\}\), and \cite[Proposition 5.2]{ParryPollicott1990} for \(L=c\mathbb Z\) with \(c>0\).
	
	The second lemma we need is the following nondegeneracy property of the limit set, which was proved  by Edwards--Lee--Oh generalizing \cite{Winter15} for rank one.
	\begin{lemma}\cite[Lemma 2.11]{ELO_anosov}\label{elo}\label{lem:not-in-submanifold}
		For any open subset $\mathcal O\subset\Fboundary$ with $\limitset\cap\mathcal O\neq\varnothing$, the set $\limitset\cap\mathcal O$ is not contained in a smooth submanifold of $\Fboundary$ of lower dimension.
	\end{lemma}
	
	\begin{proposition}[Nonconstant discrepancy criterion]
		\label{prop:non-lattice-criterion}
		Let $\psi\in\LieA^*-\{0\}$ and set ${{\mathsf K}^\psi}:=\psi\circ\ret$. If
		${{\mathsf K}^\psi}$ satisfies the nonconstant discrepancy condition, then the subset
		\[
		\psi(\lambda(\Gamma)) \subset \R
		\]
		generates a dense subgroup of $\R$.
	\end{proposition}
	
	\begin{proof}
		By \cref{lem:Jordan-proj-in-terms-of-return-vector}, together with the
		conjugacy invariance of the Jordan projection, $\{\psi(\lambda(\gamma)):\gamma\in\Gamma\}$ is precisely the set of periods
		\[
		\{({{\mathsf K}^\psi})_n(x): n \geq 1, x \in \Fix(\sigma^n)\}
		\]
		of the scalar return map
		\[
		{{\mathsf K}^\psi}:=\psi\circ\ret.
		\]
		
		Suppose, for the sake of contradiction, that the set of periods of ${{\mathsf K}^\psi}$ generates a discrete subgroup
		$L=c\mathbb Z$
		for some \(c\geq 0\). Then
		every period of \({{\mathsf K}^\psi}\) lies in \(L\). Since
		$\ret|_{\limitset}$
		is H\"older continuous with respect to the symbolic metric \(d_\vartheta\),
		so is \({{\mathsf K}^\psi}|_{\limitset}\). Hence, \cref{lem:Livsic-Schottky-shift} gives a continuous function
		\[
		\rho:\limitset\to \R/L
		\]
		such that
		\[
		{{\mathsf K}^\psi}=\rho-\rho\circ\sigma
		\qquad \pmod{L}.
		\]
		Consequently, for every \(m\geq 1\),
		\[
		({{\mathsf K}^\psi})_m=\rho-\rho\circ\sigma^m
		\qquad \pmod{L}.
		\]
		
		Let \(\zeta\) be an infinite reduced word, and let \(x,x'\in \limitset\cap\Sball_{\zeta_0}^-\).
		Then
		\[
		\begin{aligned}
			\psi(\Delta_\zeta(x,x'))
			&=
			\lim_{m\to\infty}
			\left(
			({{\mathsf K}^\psi})_m(\zeta_{[m]}^{-1}x)
			-
			({{\mathsf K}^\psi})_m(\zeta_{[m]}^{-1}x')
			\right) \\
			&=
			\lim_{m\to\infty}
			\left(
			\rho(\zeta_{[m]}^{-1}x)-\rho(x)
			-
			\rho(\zeta_{[m]}^{-1}x')+\rho(x')
			\right)
			\qquad \pmod{L} .
		\end{aligned}
		\]
		
		The inverse branches are uniformly contracting, so the two points
		\[
		\zeta_{[m]}^{-1}x
		\quad\text{and}\quad
		\zeta_{[m]}^{-1}x'
		\]
		converge to each other as \(m\to\infty\). By (uniform) continuity of \(\rho\), their
		images in \(\R/L\) also converge to each other. Therefore
		\[
		\psi(\Delta_\zeta(x,x'))
		=
		-\rho(x)+\rho(x')
		\qquad \pmod{L}.
		\]
		The same formula holds for any other infinite reduced word \(\eta\). Hence,
		on the common domain $\limitset \cap \bigl(\Sball_{\zeta_0}^- \cap \Sball_{\eta_0}^-\bigr)$, we have
		\[
		\Phi_{\zeta,\eta}^{\psi}(x,x')
		=
		0
		\qquad \pmod{L},
		\]
		or equivalently,
		\[
		\Phi_{\zeta,\eta}^{\psi}(x,x')\in L.
		\]

		Now, choose \(\zeta,\eta,x_0,x_0'\), and
		\(v\in T_{x_0}\Fboundary\), from the nonconstant discrepancy condition of ${{\mathsf K}^\psi}$. Set
		\[
		F:=\Phi_{\zeta,\eta}^{\psi}(\cdot,x_0').
		\]
		The preceding paragraph shows that
		\[
		F(\limitset\cap U)\subset L
		\]
		for some neighborhood \(U\) of \(x_0\) contained in the  domain of
		\(F\). Since \(L\) is a discrete
		subgroup of \(\R\) and \(F\) is continuous, after shrinking \(U\) to a
		smaller neighborhood of \(x_0\) if necessary, we have
		\[
		F(x)=F(x_0)
		\qquad
		\text{for all }x\in\limitset\cap U.
		\]
		Thus,
		\[
		\limitset\cap U
		\subset
		F^{-1}(F(x_0)).
		\]
		
		On the other hand, the nonconstant discrepancy condition of ${{\mathsf K}^\psi}$ gives
		\[
		DF(x_0)(v)
		=
		D_x\Phi_{\zeta,\eta}^{\psi}(x_0,x_0')(v)
		\neq 0.
		\]
		Therefore,
		the differential \(DF(x_0)\) is nonzero. Hence, by the implicit function
		theorem, after shrinking \(U\) if necessary, the level set
		\[
		F^{-1}(F(x_0))\cap U
		\]
		is a smooth hypersurface in \(\Fboundary\). In particular, it is a smooth
		submanifold of lower dimension. This contradicts
		\cref{elo}. 
		Hence, the set of periods of ${{\mathsf K}^\psi}$ are not contained in any discrete subgroup of $\R$. Since every proper closed subgroup of $\R$ is discrete, the subgroup generated by the periods of ${{\mathsf K}^\psi}$ is dense in $\R$.
	\end{proof}

	\subsection{A rational-map analogue}
	\label{subsec:rational-map-analogue}
	We briefly explain that the non-lattice criterion above is not specific to
	the linear Lie group setting. For instance, the same argument applies to expanding Markov
	systems, provided one has an appropriate nonconstant discrepancy condition and a
	nondegeneracy property of the limit set.
	We refer to \cite{OW17} for background for the discussion below.
	Let
	\[
	f:\widehat{\mathbb C}\to\widehat{\mathbb C}
	\]
	be a hyperbolic rational map of degree at least two, and let \(J=J(f)\) be its
	Julia set. Recall that hyperbolicity means that \(f\) is expanding on \(J\).
	After conjugating, we may assume that \(\infty\notin J\).
	
	Let $P_1,\ldots,P_{k_0}$
	be a sufficiently small Markov partition for \(J\). Thus \(J=\bigcup_j P_j\),
	the interiors of the \(P_j\)'s are disjoint relative to \(J\), and each \(f(P_i)\)
	is a union of partition elements. We choose open neighborhoods \(U_j\supset P_j\)
	so that, whenever \(f(P_i)\supset P_j\), there is a holomorphic inverse branch $g_{ij}:U_j\to U_i$
	of \(f\). Thus the Markov
	pieces \(P_j\), together with their neighborhoods \(U_j\), play the role of
	the Schottky domains, and the maps \(g_{ij}\) play the role of Schottky
	inverse branches (cf. \cite[\S2.2]{OW17}).
	
	The role of the Busemann return map \(\ret\) is played by the distortion
	function
	\[
	\tau(z):=\log |f'(z)|.
	\]
	For \(n\geq1\), set
	\[
	\tau_n(z):=\sum_{j=0}^{n-1}\tau(f^j z).
	\]
	If \(p\in J\) is a periodic point with \(f^n(p)=p\), then
	\[
	\tau_n(p)=\log |(f^n)'(p)|.
	\]
	Thus the periods of \(\tau\) are precisely the logarithms of the moduli of
	periodic multipliers.

	Let $\xi=(\ldots,\xi_{-2},\xi_{-1},\xi_0)$
	be a backward admissible itinerary. Following \cite[\S3.1]{OW17}, denote by $ g_\xi^n:U_{\xi_0}\to U_{\xi_{-n}}$
	the corresponding inverse branch of \(f^n\).
	For \(x,y\in U_{\xi_0}\), define
	\[
	\tau_\infty(\xi,x,y)
	:=
	\sum_{n=1}^\infty
	\bigl(
	\tau(g_\xi^n x)-\tau(g_\xi^n y)
	\bigr).
	\]
	This is the rational-map analogue of the discrepancy \(\Delta_\eta\). The
	series converges because the inverse branches are uniformly contracting and
	\(\tau\) is Lipschitz on the chosen neighborhoods.
	The series converges in \(C^\infty\) on compact subsets of \(U_{\xi_0}\times
	U_{\xi_0}\), because the inverse branches \(g_\xi^m\) are uniformly contracting
	and \(\tau\) is smooth on the chosen neighborhoods.
	
	For two backward itineraries \(\xi,\widehat \xi\) with the same
	terminal symbol \(\xi_0=\widehat\xi_0=j\), 
	set
	\[
	\Phi_{\xi,\widehat\xi}(x,y)
	:=
	\tau_\infty(\xi,x,y)-\tau_\infty(\widehat\xi,x,y),
	\qquad x,y\in U_j .
	\]

	The non-local integrability condition of Oh--Winter is exactly the analogue
	of the nonconstant discrepancy condition used above: there exist \(j\), points
	\(x_0,x_1\in P_j\), and backward admissible itineraries \(\xi,\widehat\xi\)
	ending at \(j\), such that
	\[
	D_x\Phi_{\xi,\widehat\xi}(x_1,x_0)\neq0.
	\]
	Oh--Winter prove that this condition holds for \(\tau=\log |f'|\) whenever
	\(f\) is not conjugate to \(z\mapsto z^{\pm d}\) \cite[Theorem 3.4]{OW17}.
	In particular, if \(J(f)\) is not contained in a circle, then \(f\) is not
	conjugate to a monomial, and hence the condition holds.

	We also need the following nondegeneracy property of Julia sets proved by Bergweiler--Eremenko \cite[Theorem 2]{Bergweiler_Eremenko}: if \(J(f)\)
	is not contained in a circle, then it is not locally contained in any
	smooth curve.
	
	Using the above two ingredients, we can derive the following proposition.
	
	\begin{proposition}
		\label{prop:rational-density}
		Let \(f:\widehat{\mathbb C}\to\widehat{\mathbb C}\) be a hyperbolic rational
		map whose Julia set \(J(f)\) is not contained in a circle. Then the subgroup
		of \(\mathbb R\) generated by
		\[
		\left\{
		\log |(f^n)'(p)|:
		f^n(p)=p,\ p\in J(f),\ n\geq1
		\right\}
		\]
		is dense in \(\mathbb R\). Equivalently, the moduli of the multipliers of
		repelling periodic points generate a dense subgroup of \(\mathbb R_{>0}\)
		under multiplication.
	\end{proposition}
	
	\begin{proof}
		The proof is the same as the proof of \cref{prop:non-lattice-criterion}, with
		the following replacements:
		\[
		\begin{array}{c|c}
			\text{Schottky setting} & \text{rational-map setting} \\
			\hline
			\Lambda_\Gamma & J(f) \\
			\Sball_\gamma &  U_j \\
			\sigma:\Lambda_\Gamma\to\Lambda_\Gamma & f:J(f)\to J(f) \\
			\eta_{[m]}^{-1} & g_\xi^m \\
			\ret^\psi & \tau=\log |f'| \\
			(\ret^\psi)_n & \tau_n=\log |(f^n)'| \\
			\Delta_\eta & \tau_\infty(\xi,\cdot,\cdot) \\
			\Phi_{\eta,\theta} & \Phi_{\xi,\widehat\xi}.
		\end{array}
		\]
		If the logarithmic multiplier moduli were contained in a discrete subgroup
		\(c\mathbb Z\subset\mathbb R\), then the Liv\v{s}ic theorem for the expanding
		Markov map \(f:J(f)\to J(f)\) would imply that \(\tau\) is cohomologous to
		zero modulo \(c\mathbb Z\). The same limiting argument as in
		\cref{prop:non-lattice-criterion} would then force every two-branch
		discrepancy \(\Phi_{\xi,\widehat\xi}\) to be locally constant on \(J(f)\).
		This contradicts the non-local integrability condition together with the fact
		that \(J(f)\) is not locally contained in a smooth curve.
	\end{proof}
	
	\begin{remark}
		The conclusion of \cref{prop:rational-density} follows from the prime
		orbit theorem of Oh--Winter \cite{OW17}. Indeed, their theorem gives the
		asymptotic
		\[
		\#\{\widehat x:\ |\lambda(\widehat x)|<t\}
		=
		\operatorname{Li}(t^\delta)+O(t^{\delta-\epsilon})
		\]
		for hyperbolic rational maps not conjugate to monomials, where $\lambda(\widehat x)$ denotes the multiplier of a periodic orbit $\widehat x$ and $\delta$ is the Hausdorff dimension of $J(f)$. From
		that asymptotic one can also deduce the density statement above: if all
		\(\log|\lambda(\widehat x)|\) lay in \(c\mathbb Z\), then the corresponding
		counting function would be constant on the intervals
		\((e^{kc},e^{(k+1)c}]\), contradicting the asymptotic formula. The point of
		\cref{prop:rational-density} is that this qualitative density statement follows
		directly from the nonconstant discrepancy condition and the nondegeneracy of the Julia set,
		without using the full analytic machinery of the prime orbit theorem.

		Unlike the group case, however, there is no automatic reduction from a
		general rational map to a hyperbolic rational map analogous to passing from a
		Zariski dense group to a Zariski dense Schottky subgroup. Moreover, the naive
		statement for all rational maps whose Julia set is not contained in a circle is
		false because of Latt\`es maps: their Julia set is the whole sphere, while their
		periodic multiplier moduli lie in a discrete multiplicative subgroup. On the other hand, a recent theorem of Ji--Xie--Zhang \cite{JXZ26} proves the corresponding density statement for arbitrary rational maps whose Julia set is not contained in a circle, with Latt\`es maps as the only exception.
		
	\end{remark}

	\section{Proof of the nonconstant discrepancy condition}
	\label{sec:proof-non-lattice}
	
	In this section, we prove the following proposition.
	
	\begin{proposition}\label{prop:main-non-lattice}
		For every $\psi\in\aL^*-\{0\}$, the scalar return function \({{\mathsf K}^\psi}:=\psi\circ \ret\) satisfies the nonconstant discrepancy condition.
	\end{proposition}

	The proof is geometric. We first define a \emph{higher-rank} rectangular holonomy as an $A$-displacement obtained by alternating $N$- and $\check N$-moves in the Schottky domains of the Furstenberg boundary, and establish a derivative identity. This is motivated by rank-one arguments in other contexts \cite{SW21}.
	We then relate the $A$-displacement with the two-branch discrepancy for infinite words and prove, by a Lie bracket computation between root spaces, that every nonzero scalar projection has a nonzero derivative on the limit set, which is the nonconstant discrepancy condition; see \cref{nli2}.

	For $\gamma \in \gen$, write
	\[
	x_\gamma:=\gamma\gamma\cdots\in {\mathbf b}_\gamma,
	\qquad
	y_\gamma:=\gamma^{-1}\gamma^{-1}\cdots\in {\mathbf b}_{\gamma^{-1}}.
	\]
	Then, choose a corresponding frame $g_\gamma$, i.e., any element $g_\gamma \in G$ with forward and backward endpoints
	\[
	g_\gamma^+=x_\gamma,
	\qquad
	g_\gamma^-=y_\gamma.
	\]
	Fix a distinguished letter $\omega \in \gen$ throughout this section. We call $g_\omega$ the base frame. By $G$-equivariance of the entire setting, we may work in coordinates in which
	\[
	x_\omega=e^+,
	\qquad
	y_\omega=e^-,
	\]
	and we may assume that our choice of base frame is $g_\omega = e \in G$.

	Define the smooth visual charts
	\[
	\vis^+:\check N\to\{\zeta\in\F:(\zeta,e^-)\in\F^{(2)}\},
	\qquad h\mapsto h^+,
	\]
	and
	\[
	\vis^-:N\to\{\eta\in\F:(e^+,\eta)\in\F^{(2)}\},
	\qquad n\mapsto n^-.
	\]
	For $\gamma \in \gen$, recall the notation ${\mathbf b}_\gamma^-:={\mathbf b}^{(1)}- {\mathbf b}_\gamma$. The Schottky separation condition implies
	\[
	{\mathbf b}_\gamma\times {\mathbf b}_\gamma^-\subset \F^{(2)}.
	\]
	Thus, we may set
	\[
	\check N_\gamma:=(\vis^+ \circ g_\gamma)^{-1}({\mathbf b}_\gamma),
	\qquad
	N_\gamma:=(\vis^- \circ g_\gamma)^{-1}({\mathbf b}_\gamma^-).
	\]
	Consequently,
	\[
	\vis^+_\gamma := (\vis^+ \circ g_\gamma)|_{\check N_\gamma} \quad \text{ and } \quad  \vis^-_\gamma := (\vis^- \circ g_\gamma)|_{N_\gamma}
	\]
	are diffeomorphisms onto $\Sball_\gamma$ and $\Sball_\gamma^-$ respectively.
	It is convenient to define
	\[
	h_\gamma :=(\vis^+_\gamma)^{-1}, \qquad h := h_\omega.
	\]
	
	Denote the $A$-projection by $\pi_A:MA\to A$; its differential at $e$ is $\pi_\mathfrak{a}$.
	
	\begin{figure}
		\centering
		\includegraphics{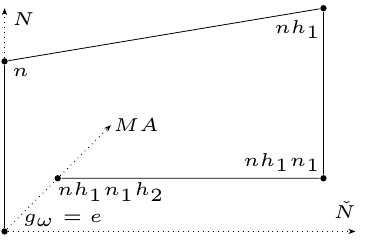}
		\caption{}
		\label{fig:RectangularHolonomy}
	\end{figure}
	
	\begin{definition}[Rectangular holonomy]\label{def:rectangle-holonomy}
		For $h\in\check N_\omega$ and $n\in N_\omega$, set
		\[
		x:=h^+,
		\qquad
		y:=n^-.
		\]
		Choose the unique elements
		\[
		h_1\in\check N,
		\qquad
		n_1\in N,
		\qquad
		h_2\in\check N
		\]
		such that
		\[
		(nh_1)^+=x,
		\qquad
		(nh_1n_1)^-=y_\omega,
		\qquad
		(nh_1n_1h_2)^+=x_\omega.
		\]
		Define
		\[
		{\mathcal H}(h,n):=\log \pi_A  (nh_1n_1h_2)\in\aL.
		\]
		See \cref{fig:RectangularHolonomy}.
	\end{definition}
	
	\begin{remark}
		In the above definition, recalling that right translation by $N$ preserves the positive endpoint, and right translation by $\check N$ preserves the negative endpoint, we also have
		\[
		(nh_1)^-=y,
		\qquad
		(nh_1n_1)^+=x,
		\qquad
		(nh_1n_1h_2)^-=y_\omega.
		\]
		Consequently, $(nh_1n_1h_2)^\pm = e^\pm$ and indeed $nh_1n_1h_2 \in MA$.
	\end{remark}
	
	We shall compare the Busemann return map with this rectangular holonomy. The comparison contains an endpoint correction. This correction is independent of the backward branch and hence cancels out in the two-branch discrepancy.
	
	For $\gamma \in \gen$, the map on $\Sball_\gamma$ given by
	\[
	z \mapsto g_\gamma h_\gamma(z)
	\]
	is a smooth lift of $\Sball_\gamma$ into $G$ with fixed negative endpoint $y_\gamma$:
	\[
	(g_\gamma h_\gamma(z))^+=z,
	\qquad
	(g_\gamma h_\gamma(z))^-=y_\gamma.
	\]
	Define the Busemann height of this lift $\kappa: \Sball^{(1)} \to \LieA$ piecewise by
	\[
	\kappa(z):=\beta_z(e,g_\gamma h_\gamma(z)),
	\qquad z\in \Sball_\gamma, \quad \gamma \in \gen.
	\]
	Since the first-level Schottky domains are disjoint, $\kappa$ is a smooth map.
	
	We shall use the following elementary relation between the Busemann function and $A$-coordinates. If $g,h\in G$ have the same positive endpoint $\xi$, then $g^{-1}h\in P=MAN$, and
	\begin{equation}\label{eq:busemann-A-coordinate}
		\log\pi_A(g^{-1}h)=\beta_\xi(g,h).
	\end{equation}
	This follows directly from the definition of the Iwasawa cocycle; the equality is unaffected by right translation of $g$ or $h$ by elements of $M$. 
	
	\begin{lemma}[Finite holonomy comparison]\label{lem:finite-holonomy}
		Let $m\ge1$ and let
		\[
		\gamma=\gamma_0\gamma_1\cdots\gamma_{m-1}\in\cA^{(m)}
		\]
		be such that $\gamma\omega$ is reduced. Let $n_\gamma\in N_\omega$ be determined by
		\[
		n_\gamma^-:=\gamma^{-1}y_{\gamma_0}.
		\]
		Then, for $x\in {\mathbf b}_\omega$,
		\begin{equation}\label{eq:finite-holonomy}
			\ret_m(\gamma x)-\ret_m(\gamma x_\omega)
			={\mathcal H}(h(x),n_\gamma)-\kappa(x)+\kappa(x_\omega)
			+\kappa(\gamma x)-\kappa(\gamma x_\omega).
		\end{equation}
	\end{lemma}
	
	\begin{proof}
		For $z\in {\mathbf b}_\omega$, the cocycle identity gives
		\[
		\ret_m(\gamma z)=\beta_{\gamma z}(e,\gamma).
		\]
		Let
		\[
		p_z:=\gamma^{-1}g_{\gamma_0}h_{\gamma_0}(\gamma z).
		\]
		Then
		\[
		p_z^+=z,
		\qquad
		p_z^-=\gamma^{-1}y_{\gamma_0}=n_\gamma^-.
		\]
		The frame $h_\omega(z)=h(z)$ has endpoints $(z,y_\omega)$. Thus $p_z$ and $h(z)$ have the same positive endpoint $z$. By \eqref{eq:busemann-A-coordinate} and $G$-equivariance of the Busemann function,
		\begin{align*}
			\log\pi_A\bigl(p_z^{-1}h(z)\bigr)
			=\beta_z(p_z,h(z))
			=\beta_{\gamma z}(g_{\gamma_0}h_{\gamma_0}(\gamma z),\gamma h(z)).
		\end{align*}
		By additivity of the Busemann function,
		\[
		\beta_{\gamma z}(g_{\gamma_0}h_{\gamma_0}(\gamma z),\gamma h(z))
		=\beta_{\gamma z}(e,\gamma)
		+\beta_z(e,h(z))
		-\beta_{\gamma z}(e,g_{\gamma_0}h_{\gamma_0}(\gamma z)).
		\]
		Putting the two calculations together and using definitions, we arrive at the identity
		\[
		\log\pi_A\bigl(p_z^{-1}h(z)\bigr) = \ret_m(\gamma z)+\kappa(z)-\kappa(\gamma z).
		\]
		The rectangular holonomy in Definition \ref{def:rectangle-holonomy} computes the change of this $A$-coordinate as $z$ moves from $x_\omega$ to $x$ while the negative endpoint $n_\gamma^-$ is fixed. Indeed, the rectangle starts from the frame with endpoints
		\((x_\omega,y_\omega)\), changes the negative endpoint to \(n_\gamma^-\),
		then changes the positive endpoint from \(x_\omega\) to \(x\), and finally
		returns the negative and positive endpoints to \(y_\omega\) and \(x_\omega\).
		Thus its \(A\)-component is precisely the difference between the two
		\(A\)-coordinates \(\log\pi_A(p_z^{-1}h(z))\) at \(z=x\) and \(z=x_\omega\). Therefore
		\[
		{\mathcal H}(h(x),n_\gamma)
		=\log\pi_A(p_x^{-1}h(x))-
		\log\pi_A(p_{x_\omega}^{-1}h(x_\omega)).
		\]
		Substituting the preceding identity for $z=x$ and $z=x_\omega$ gives
		\begin{align*}
			{\mathcal H}(h(x),n_\gamma)
			&=\ret_m(\gamma x)-\ret_m(\gamma x_\omega)
			+\kappa(x)-\kappa(x_\omega)
			-\kappa(\gamma x)+\kappa(\gamma x_\omega).
		\end{align*}
		Rearranging yields \eqref{eq:finite-holonomy}.
	\end{proof}
	
	\begin{lemma}[Two-branch holonomy identity]\label{lem:two-branch-holonomy}
		Let $\xi$ be an infinite reduced word beginning with $\omega^{-1}$. Let $n_\xi\in N_\omega$ be determined by
		\[
		n_\xi^-=
		\xi\in\Lambda_\Gamma\cap {\mathbf b}_{\omega^{-1}}.
		\]
		Then, for $x\in {\mathbf b}_\omega$,
		\begin{equation}\label{eq:Delta-H}
			\Delta_\xi(x,x_\omega)={\mathcal H}(h(x),n_\xi)-\kappa(x)+\kappa(x_\omega).
		\end{equation}
		Consequently, for any two such words $\zeta$ and $\eta$,
		\begin{equation}\label{eq:Phi-H}
			\Phi_{\zeta,\eta}(x,x_\omega)
			={\mathcal H}(h(x),n_\zeta)-{\mathcal H}(h(x),n_\eta).
		\end{equation}
	\end{lemma}
	
	\begin{proof}
		Write
		\[
		\gamma_m:=\xi_{[m]}^{-1}.
		\]
		Then $\gamma_mx$ and $\gamma_mx_\omega$ both lie in the domain of $\ret_m$. The first letter of $\gamma_m$ is $\xi_{m-1}^{-1}$. Define $n_{\gamma_m}\in N_\omega$ by
		\[
		n_{\gamma_m}^-:=\gamma_m^{-1}y_{\xi_{m-1}^{-1}}
		=\xi_{[m]}y_{\xi_{m-1}^{-1}}.
		\]
		The infinite word representing $\xi_{[m]}y_{\xi_{m-1}^{-1}}$ has the prefix $\xi_0\xi_1\cdots\xi_{m-1}$, and hence these points converge to $\xi$. Therefore
		\[
		n_{\gamma_m}\to n_\xi
		\]
		in the visual chart of $N_\omega$.
		
		By Lemma \ref{lem:finite-holonomy}, for all sufficiently large $m$,
		\begin{multline*}
			\ret_m(\gamma_mx)-\ret_m(\gamma_mx_\omega) 
			={\mathcal H}(h(x),n_{\gamma_m})-\kappa(x)+\kappa(x_\omega)   +\kappa(\gamma_mx)-\kappa(\gamma_mx_\omega).
		\end{multline*}
		The two points $\gamma_mx$ and $\gamma_mx_\omega$ lie in the same Schottky domain and converge exponentially to each other. Since $\kappa$ has bounded derivatives on each first-level domain,
		\[
		\kappa(\gamma_mx)-\kappa(\gamma_mx_\omega)\to0.
		\]
		Passing to the limit gives \eqref{eq:Delta-H}. Taking the difference between the corresponding identities for $\zeta$ and $\eta$ cancels the endpoint correction $-\kappa(x)+\kappa(x_\omega)$ and gives \eqref{eq:Phi-H}.
	\end{proof}
	
	For fixed $n\in N_\omega$, write
	\[
	{\mathcal H}_n(h):={\mathcal H}(h,n).
	\]
	Define a smooth map
	\[
	\Psi_n:\check N_\omega\to\check N
	\]
	by requiring
	\[
	(n\Psi_n(h))^+=h^+.
	\]
	Thus $\Psi_e(h)=h$, and
	\[
	(D\Psi_n)_e:\check{\mathfrak{n}}\to\check{\mathfrak{n}}
	\]
	is an isomorphism for all $n \in N_\omega$. 
	
	\begin{lemma}[Derivative of the rectangular holonomy]\label{lem:derivative-holonomy}
		For $n\in N_\omega$,
		\[
		( D {\mathcal H}_n)_e=\pi_{\aL}\circ \Ad_n|_{\check{\mathfrak{n}}}\circ (D\Psi_n)_e.
		\]
		
	\end{lemma}
	
	\begin{proof}
		Fix $n\in N_\omega$. For $h\in\check N_\omega$, the first $\check N$-move in the rectangle is
		\[
		h_1=\Psi_n(h).
		\]
		By Definition \ref{def:rectangle-holonomy}, there are unique smooth functions
		\[
		n_1(h)\in N,
		\qquad
		h_2(h)\in\check N
		\]
		such that
		\[
		q(h):=n\Psi_n(h)n_1(h)h_2(h)\in MA.
		\]
		Then
		\[
		{\mathcal H}_n(h)=\log\pi_Aq(h).
		\]
		At $h=e$, we have $\Psi_n(e)=e$, $n_1(e)=n^{-1}$, $h_2(e)=e$, and hence $q(e)=e$.
		
		Let $W\in\check{\mathfrak{n}}$ and set $h(t)=\exp(tW)$. Put
		\[
		V:=(D\Psi_n)_e(W)\in\check{\mathfrak{n}}.
		\]
		Then
		\[
		\Psi_n(h(t))=\exp(tV+O(t^2)).
		\]
		For brevity, write
		\[
		n_1(t):=n_1(h(t)),
		\qquad
		h_2(t):=h_2(h(t)),
		\qquad
		q(t):=q(h(t)).
		\]
		Thus
		\[
		q(t)=n\Psi_n(h(t))n_1(t)h_2(t)\in MA,
		\qquad q(0)=e.
		\]
		Write the first-order expansions
		\[
		q(t)=\exp(tQ+O(t^2)),\quad Q\in\mathfrak{m}\oplus\aL,
		\]
		\[
		h_2(t)=\exp(tS+O(t^2)),\quad S\in\check{\mathfrak{n}}.
		\]
		Since $n_1(0)=n^{-1}$, define
		\[
		d(t):=n_1(t)^{-1}n^{-1}\in N.
		\]
		Then
		\[
		d(t)=\exp(tD+O(t^2)),\quad D\in\mathfrak{n}.
		\]
		From
		\[
		q(t)=n\Psi_n(h(t))n_1(t)h_2(t),
		\]
		we multiply on the right by $h_2(t)^{-1}n_1(t)^{-1}n^{-1}$ and get
		\[
		n\Psi_n(h(t))n^{-1}=q(t)h_2(t)^{-1}d(t).
		\]
		The left-hand side is
		\[
		n\exp(tV+O(t^2))n^{-1}=\exp(t\Ad_nV+O(t^2)).
		\]
		The right-hand side is
		\[
		\exp(tQ+O(t^2))\exp(-tS+O(t^2))\exp(tD+O(t^2)).
		\]
		Comparing first-order terms gives
		\[
		\Ad_nV=Q-S+D.
		\]
		Projecting to $\aL$, the terms $S\in\check{\mathfrak{n}}$ and $D\in\mathfrak{n}$ vanish. Hence
		\[
		\pi_{\aL}(\Ad_nV)=\pi_{\aL}(Q).
		\]
		Since $q(t)\in MA$ and $A$ commutes with $M$, the $\aL$-component of $Q$ is precisely
		\[
		\left.\frac{d}{dt}\right|_{t=0}\log\pi_Aq(t).
		\]
		Therefore
		\[
		(D {\mathcal H}_n)_e(W)=\pi_{\aL}\bigl(\Ad_n (D\Psi_n)_e(W)\bigr).
		\]
		This proves the claim.
	\end{proof}

	See \cite[Chapter VI, \S 3, Lemma 3.1]{Hel01} for the following lemma.
	\begin{lemma}[Root space nondegeneracy] \label{lem:root-nondeg}
		Let $\alpha\in\Pi$. There exist $E_\alpha\in\g_\alpha$ and $E_{-\alpha}\in\g_{-\alpha}$ such that
		\[  [E_\alpha,E_{-\alpha}]\in \R_{>0}H_\alpha.
		\]
	\end{lemma}

	\begin{lemma}\label{lem:nonzero-derivative}
		Let $\psi\in\aL^*-\{0\}$. Then, there exist an infinite reduced word $\zeta$ beginning with $\omega^{-1}$ and a tangent vector $v\in T_{x_\omega}\F$ such that
		\[
		\psi\left(D_x|_{x=x_\omega}\bigl({\mathcal H}(h(x),n_\zeta)\bigr)(v)\right)\ne0.
		\]
	\end{lemma}
	
	\begin{proof}
		Choose $\alpha\in\Pi$ such that $\psi(H_\alpha)\ne0$; we may do so since $\{H_\alpha\}_{\alpha \in \Pi}$ forms a basis of $\LieA$. By Lemma \ref{lem:root-nondeg}, we may choose $0\ne Y = E_\alpha \in\g_\alpha$ and $0 \ne Z = E_{-\alpha}\in\g_{-\alpha}$; then
		\[
		\psi([Y,Z])\ne0.
		\]
		Define
		\[
		L:N\to\R,
		\qquad
		L(n):=\psi\bigl(\pi_{\aL}(\Ad_n Z)\bigr).
		\]
		Then, $L(e) = 0$ and
		\[
		D L_e(Y)=\psi([Y,Z])\ne0.
		\]
		Hence, in a sufficiently small neighborhood of $e$ in $N$, the zero set \(L^{-1}(0)\) is contained in a proper smooth hypersurface through $e$.
		Its image under $\vis_\omega^-$ is then locally contained in a smooth hypersurface through \(y_\omega\). By
		\cref{elo}, the set \(\Lambda_\Gamma\) is not
		locally contained in this hypersurface. Hence, we may choose
		\(n\in N_\omega\) such that
		\[
		n^-\in\Lambda_\Gamma\cap {\mathbf b}_{\omega^{-1}},
		\qquad
		L(n)\ne0.
		\]
		Let $\zeta$ be the infinite reduced word representing $n^-$. Then $\zeta$ begins with $\omega^{-1}$ and $n=n_\zeta$. Since $(D\Psi_n)_e:\check{\mathfrak{n}}\to\check{\mathfrak{n}}$ is an isomorphism for $n \in N_\omega$, choose $W\in\check{\mathfrak{n}}$ such that
		\[
		(D\Psi_n)_e(W)=Z.
		\]
		By Lemma \ref{lem:derivative-holonomy},
		\[
		(D {\mathcal H}_n)_e(W)=\pi_{\aL}(\Ad_nZ).
		\]
		Applying $\psi$ gives
		\[
		\psi((D{\mathcal H}_n)_e(W))=\psi\bigl(\pi_{\aL}(\Ad_nZ)\bigr)=L(n)\ne0.
		\]
		Finally, recalling $\vis^+_\omega(e) = x_\omega$, let
		\[
		v:=(D\vis^+_\omega)_e(W)\in T_{x_\omega}\F.
		\]
		Since $h=(\vis^+_\omega)^{-1}$, the desired inequality follows.
	\end{proof}
	
	\begin{proof}[Proof of Proposition \ref{prop:main-non-lattice}]
		Let $\psi\in\aL^*-\{0\}$. By Lemma \ref{lem:nonzero-derivative}, choose an infinite reduced word $\zeta$ beginning with $\omega^{-1}$ and a tangent vector $v\in T_{x_\omega}\F$ such that
		\[
		\psi\left(D_x|_{x=x_\omega}\bigl({\mathcal H}(h(x),n_\zeta)\bigr)(v)\right)\ne0.
		\]
		Let
		\[
		\eta:=\omega^{-1}\omega^{-1}\omega^{-1}\cdots.
		\]
		Then $\eta$ begins with $\omega^{-1}$ and $n_\eta=e$. Moreover,
		\[
		{\mathcal H}(h,e)=0 \quad\text{for all $h\in\check N_\omega$.}
		\]
		By Lemma \ref{lem:two-branch-holonomy},
		\[
		\Phi^\psi_{\zeta,\eta}(x,x_\omega)=\psi\bigl({\mathcal H}(h(x),n_\zeta)\bigr)
		\]
		for $x$ sufficiently close to $x_\omega$. Differentiating at $x=x_\omega$ in the direction $v$ gives
		\[
		D_x\Phi^\psi_{\zeta,\eta}(x_\omega,x_\omega)(v)
		=\psi\left(D_x|_{x=x_\omega}\bigl({\mathcal H}(h(x),n_\zeta)\bigr)(v)\right)
		\ne0.
		\]
		Since $x_\omega\in\Lambda_\Gamma$, this is exactly the nonconstant discrepancy condition for ${{\mathsf K}^\psi}=\psi\circ \ret$.
	\end{proof}
	
	\section{Conclusion} 
	
	\begin{proof}[Proof of Theorem \ref{m2}]
		By Lemma \ref{elo2}, choose a Zariski dense Schottky  subgroup
		$\Gamma_0<\Gamma$.
		Applying Proposition \ref{prop:main-non-lattice} and Proposition \ref{prop:non-lattice-criterion} to $\Gamma_0$, we obtain that the subgroup of $\R$ generated by
		$\{\psi(\lambda(\gamma)):\gamma\in\Gamma_0\}$
		is dense. Since $\lambda(\Gamma_0)\subset\lambda(\Gamma)$, the subgroup generated by $\psi(\lambda(\Gamma))$ contains a dense subgroup. It is therefore dense in $\R$.
	\end{proof}

	\begin{proof}[Proof of \Cref{m1} assuming \Cref{m2}]
		Let $\fa_0:=\overline{\langle \lambda(\Gamma)\rangle}$
		be the closure of the subgroup of \(\fa\) generated by \(\lambda(\Gamma)\).
		We first show that \(\fa_0=\fa\).
		
		Recall that a closed subgroup of a finite-dimensional real vector space is of
		the form $\fa_0=V\oplus \Lambda$
		where \(V<\fa\) is a vector subspace and \(\Lambda\) is a lattice in a
		complementary vector subspace. Equivalently, after choosing suitable vectors
		\(v_i,w_j\in\fa\), we may write
		\[
		\fa_0
		=
		\bigoplus_i \mathbb R v_i
		\oplus
		\bigoplus_j \mathbb Z w_j .
		\]
		
		Suppose first that the lattice part is nontrivial. Choose one of the lattice
		generators, say \(w_{j_0}\), and let $\psi:\fa\to\mathbb R$
		be a linear form which vanishes on \(V\) and on all \(w_j\) with
		\(j\neq j_0\), and satisfies $\psi(w_{j_0})=1$.
		Then $\psi(\fa_0)\subset \mathbb Z$.
		Since \(\lambda(\Gamma)\subset\fa_0\), it follows that $\psi(\lambda(\Gamma))\subset \mathbb Z$,
		contradicting \Cref{m2}. Therefore the lattice part is trivial, and
		\(\fa_0\) is a vector subspace of \(\fa\).
		If \(\fa_0\neq\fa\), choose a nonzero linear form $ \psi\in\fa^*$
		such that \(\psi|_{\fa_0}=0\). Then
		$    \psi(\lambda(\Gamma))=\{0\}$,
		again contradicting \Cref{m2}. Hence
		$    \fa_0=\fa$
		and thus \(\lambda(\Gamma)\) generates a dense subgroup of \(\fa\).
		
		Now suppose, for
		contradiction, that $\operatorname{int}(\L_\Gamma)=\emptyset$.
		Since \(\L_\Gamma\) is a convex cone by Benoist \cite{Benoist97}, it is contained in a proper linear subspace of \(\fa\). Hence there
		exists a nonzero linear form $\psi\in\fa^*$
		such that $    \psi|_{\L_\Gamma}=0$. Since $\lambda(\Gamma)\subset \L_\Gamma$,
		we get
		\[
		\psi(\lambda(\Gamma))=\{0\},
		\]
		contradicting \Cref{m2}. Therefore $\operatorname{int}(\L_\Gamma)\neq\emptyset $.
	\end{proof}

	\begin{remark}
		Since the rectangular holonomy is a priori $MA$-valued, the same techniques can prove the non-arithmeticity of the generalized Jordan spectrum in $MA$ (see \cite{LO24}).
	\end{remark}
	\bibliographystyle{alpha}
	\bibliography{NA}
	
\end{document}